\documentclass{amsart}

\usepackage{amssymb}
\usepackage{amsmath}

\usepackage{hyperref}

\usepackage{graphicx}

\newtheorem{theorem}{Theorem}[section]

\newtheorem{lemma}[theorem]{Lemma}
\newtheorem{proposition}[theorem]{Proposition}

\theoremstyle{definition}
\newtheorem{definition}[theorem]{Definition}
\newtheorem{remark}[theorem]{Remark}

\numberwithin{equation}{section}

\begin{document}

\title[Schwarz-Pick lemma for  $(\alpha, \beta)$-harmonic functions in the unit disc]{Schwarz-Pick lemma for  $(\alpha, \beta)$-harmonic functions \\ in the unit disc}
\author[M. Arsenovi\'c]{Milo\v s Arsenovi\'c}
\address{Department of Mathematics\\ University of Belgrade\\
Studentski Trg 16\\
11000 Belgrade, Serbia}
\email{arsenovic@matf.bg.ac.rs}
\author[J. Gaji\'c]{Jelena Gaji\'c}
\address{Faculty of Natural Sciences and Mathematics \\ University of Banja Luka\\ Mladena Stojanovi\'ca 2\\ 78000 Banja Luka, Bosnia and Herzegovina}
\email{jelena.gajic.mm@gmail.com}

\

\date{}

\begin{abstract}
We obtain Schwarz–Pick lemma for $(\alpha, \beta)$-harmonic functions $ u $  in the disc, where $ \alpha$ and  $\beta $ are complex parameters satisfying 
$ \Re \alpha + \Re \beta>-1$ . We prove sharp estimate of $\| Du(0) \|$ for such functions in terms of $ L^p $ norm of the boundary function. Also, we give asymptotically sharp estimate of $\lVert D u(z) \rVert$. Estimates of $ \lVert Du(z)\rVert $ are extended to higher order derivatives. Our results extend earlier results for $ \alpha$-harmonic and $ T_\alpha$-harmonic functions. 
\end{abstract}

\subjclass[2020]{30C80,31A05, 31A10, 35J25}

\keywords{$(\alpha, \beta)$-harmonic functions, Schwarz lemma, Poisson integral}

\maketitle

\section{Introduction and notation}\label{Intro}

The open unit disc in the complex plane is denoted by $ \mathbb{D} $ and its boundary by $\partial \mathbb{D}=\mathbb T$. 
We normalize standard measure on $\mathbb T$, so we have $m(\mathbb T)=1$.
We will use standard first order defferential operators

$$
\partial_z = \frac{\partial}{\partial z} = \frac 12\left(\frac{\partial}{\partial x} - i \frac{\partial}{\partial y}\right)\quad \mbox{and} \quad   
\partial_{\overline z} =  \frac{\partial}{\partial \overline z} = \frac 12\left(\frac{\partial}{\partial x} + i \frac{\partial}{\partial y}\right).
$$

The derivative of a $C^1$ complex valued function $f$ on a domain $\Omega \subset \mathbb C$ at a point $z$ in $\Omega$ is
denoted by $Df(z)$. The norm of this linear map on $\mathbb R^2$ is denoted by $\Vert Df(z) \Vert$ and we have

\begin{equation}\label{normDf} 
\lVert Df(z)\rVert=\sup\{\vert Df(z)\zeta\vert: \vert \zeta\vert=1\}=\vert f_z(z)\vert +\vert f_{\overline z}(z)\vert.
\end{equation}
 
Let, for $ \alpha, \beta\in \mathbb{C}$ 
$$ 
L_{\alpha, \beta} =(1-\vert z\vert^2) \left((1-\vert z\vert^2) \frac{\partial^2 }{\partial z \partial \overline{z}}
+ \alpha z \frac{\partial }{\partial z}
+ \beta \overline{z}\frac{\partial }{\partial \overline{z}}-\alpha \beta\right)
$$
be a second order uniformly elliptic linear partial differential operator in the unit disc $\mathbb D$. These operators have been introduced and studied in higher dimensions by Geller \cite{Ge} and Ahern et al. \cite{ABC} and in the planar case in \cite{KO}. They are degenerate near the boundary of the unit disc (the unit ball in the multidimensional case, see \cite{ABC}), except for the special case $ \alpha=\beta=0$ when  $L_{\alpha, \beta} = 4^{-1} (1 - \vert z \vert^2)^2 \Delta$, where $\Delta$ is the Laplace operator with respect to $x$ and $y$ variables. 
 
A function $ u $ in $C^2( \mathbb D)$ is  said to be $ (\alpha, \beta)$-harmonic if $L_{\alpha, \beta} u=0$. The vector space of all such functions is denoted by
$h_{\alpha, \beta}(\mathbb D)$.  If $ \alpha>-1 $ the $ (0,\alpha) $-harmonic 
functions are called $ \alpha$-harmonic and if $ \alpha\in \mathbb{R} $ the $ (\frac{\alpha}{2},\frac{\alpha}{2}) $-harmonic 
functions are called $ T_\alpha$-harmonic functions. These one parameter cases  appeared in \cite{OW,MA,LRW,LC2022,LIet, KM2023hpnorm} and \cite{OP}.

Colonna \cite{Colonna} obtained Scwarz-Pick lemma for complex-valued harmonic function. More precisely, if $ f $ is a complex-valued harmonic functions from $ \mathbb{D} $ into itself, then

$$ \lVert Df(z)\rVert\leq \frac{4}{\pi}\frac{1}{1-\vert z\vert^2}, \qquad z \in \mathbb D.$$
 The following results can be interpreted as Schwarz-Pick lemma for $ (0,\alpha)$-harmonic functions.

\begin{theorem}(\cite{LIet}, Theorem 1.1.)\label{Teorem LI} Suppose that $ f \in C(\overline{\mathbb D})$ is an $ \alpha$-harmonic function in $ \mathbb{D} $ with $ \alpha>-1$,  and that 
$ \sup_{z\in \overline{\mathbb{D}}}\vert f(z)\vert\leq M, $ where $ M $ is a constant. Then for $ z\in \mathbb{D}, $ 
$$ \lVert Df(z)\rVert\leq {M(\alpha+2)}\frac{\Gamma(\alpha+1)}{\Gamma^2\left(\frac{\alpha}{2}+1\right)}\frac{1}{1-\vert z\vert}\leq {2M(\alpha+2)}\frac{\Gamma(\alpha+1)}{\Gamma^2\left(\frac{\alpha}{2}+1\right)}\frac{1}{1-\vert z\vert^2}.$$
In particular, if $ f $ maps $ \mathbb{D} $ into $ \mathbb{D}, $ then
 $$ \lVert Df(z)\rVert\leq {2(\alpha+2)}\frac{\Gamma(\alpha+1)}{\Gamma^2\left(\frac{\alpha}{2}+1\right)}\frac{1}{1-\vert z\vert^2}.$$
\end{theorem}

\begin{lemma}(\cite{LRW,LC2022})(Lemma 3.3, Lemma 4.3)\label{dva rada}
 Assume that $\varphi\in C(\mathbb T)$ and let $ \alpha>-1$. Then

$$ \lVert DP_{0,\alpha}[\varphi](z)\rVert\leq\left\{
\begin{array}{lc}
2^{1-\alpha}\lVert \varphi\rVert_\infty(1-\vert z\vert^2)^{\alpha-1},& -1<\alpha<0,\\
(1+\alpha)2^{1+\alpha}\lVert \varphi \rVert_\infty\frac{1}{1-\vert z\vert^2},& \alpha\geq 0.
\end{array}\right.$$
\end{lemma}

We recall definition of the Pochhammer symbol: 

$$
(a)_0 = 1 \qquad \mbox{and} \qquad  (a)_k = a (a+1) \cdots (a+k-1)\quad \mbox{\rm for} \quad k = 1, 2, \ldots.
$$

\section{Auxiliary results}\label{ar}

The function $u_{\alpha, \beta}$, where $\alpha$ and $\beta$ are complex numbers, defined by
\begin{equation}\label{ualphabeta}
 u_{\alpha, \beta}(z)= 
 \frac{(1-\vert z\vert^2)^{\alpha+\beta+1}}{(1-z)^{\alpha+1}
(1-\overline{z} )^{\beta+1}}, \qquad \vert z \vert < 1
\end{equation} 
plays important role in the theory of $(\alpha, \beta)$-harmonic functions, see \cite{KO}.

\begin{theorem} [\cite{KO}, Theorem 1.4 and Theorem 6.4.]  \label{Uab}
Let $\alpha, \beta \in \mathbb C$. Then  $u_{\alpha, \beta} \in h_{\alpha, \beta}(\mathbb D)$ and 
\begin{equation}\label{procuab}
\vert u_{\alpha, \beta}(z)\vert \leq  e^{\frac{\pi}{2}
\vert \Im  \alpha-\Im \beta\vert}\frac{(1-\vert z\vert^2)^{\Re \alpha + \Re \beta+1}}{\vert 1-z\vert^{\Re \alpha + \Re \beta+2}}, \quad \vert z\vert<1.
\end{equation}
 Moreover, if $\Re \alpha + \Re \beta>-1$, the following estimate holds:
\begin{equation}\label{L1est}
\frac{1}{2\pi} \int_{-\pi}^\pi \vert u_{\alpha, \beta}(r e^{i\theta})\vert d\theta \leq e^{\frac{\pi}{2}
\vert \Im  \alpha-\Im \beta\vert} 
\frac{\Gamma(\Re \alpha + \Re \beta +1)}{\Gamma^2\left(\frac{\Re \alpha + \Re \beta}{2} +1\right)}, \qquad 0 \leq r < 1.
\end{equation}
\end{theorem}

In the following we assume that  $\alpha, \beta \in \mathbb C \setminus \mathbb Z^-  = \mathbb C \setminus \{-1, -2, \ldots \}$ satisfy 
$\Re \alpha + \Re \beta > -1$. The $(\alpha,\beta)$-harmonic Poisson kernel is defined by
\begin{equation}\label{valphabeta}
P_{\alpha, \beta}(z,\zeta) = c_{\alpha, \beta} u_{\alpha, \beta}(z\overline\zeta ), \qquad z \in \mathbb D, \quad \zeta \in \mathbb{T}
\end{equation}
where a normalizing constant $c_{\alpha, \beta}$ is given by
\begin{equation}\label{coef}
 c_{\alpha, \beta}=\frac{\Gamma(\alpha+1)\Gamma(\beta+1)}{\Gamma(\alpha + \beta + 1)}.
\end{equation}

If $ \Re\alpha=\alpha=\beta>-\frac{1}{2} $   then $ P_{\alpha, \beta}$ is real and positive.

\begin{definition}\label{Poissonext}
For $f \in L^1(\mathbb T)$ the $ (\alpha,\beta)$-Poisson integral of $f$ is defined by 
\begin{equation}\label{poissonL1}
P_{\alpha, \beta}[f](z) = \int_{\mathbb T} P_{\alpha, \beta}(z, \zeta) f(\zeta) dm(\zeta), \qquad z \in \mathbb D.
\end{equation}
\end{definition}

We use the same notation $P_{\alpha, \beta}$ for the above operator and its integral kernel, this should cause no confusion.
Since $P_{\alpha, \beta}(z, \zeta)$ is  $(\alpha, \beta)$-harmonic 
with respect to $z \in \mathbb D$ for each fixed $\zeta \in \mathbb T$ it follows that 

$$ P_{\alpha, \beta} [f] \in h_{\alpha, \beta}(\mathbb D), \qquad f \in L^1(\mathbb T).$$

\begin{proposition}(\cite{MA})\label{MA}
Let $ m\in \mathbb{R} $ and $ s>\frac 12$. Then for all $ z\in \mathbb{D}$, we have 
$$ \frac{1}{2\pi}\int_0^{2\pi}\frac{(1-\vert z\vert^2)^m}{\vert 1-ze^{-i\theta}\vert^{2s}}d\theta\leq \frac{\Gamma(2s-1)}{\Gamma^2(s)}(1-\vert z\vert^2)^{m-2s+1}.$$
\end{proposition}

For the following Lemma see Lemma 1.1 and Lemma 1.2 from \cite{KO}.

\begin{lemma}\label{partialu}
Let $ u_{\alpha, \beta} $ be as in \eqref{ualphabeta} for some $ \alpha, \beta\in \mathbb{C}$. Then 

$$ 
\frac{\partial u_{\alpha, \beta}}{\partial z}(z) = \left( -(\alpha+\beta+1) \frac{\overline z}{1-\vert z\vert^2} + 
\frac{\alpha + 1}{1-z}\right) u_{\alpha, \beta}(z) 
$$
and 
$$ \frac{\partial u_{\alpha, \beta}}{\partial \overline z}(z) = \left( -(\alpha+\beta+1) \frac{z}{1-\vert z\vert^2} + 
 \frac{\beta + 1}{1-\overline z} \right) u_{\alpha, \beta}(z). $$
\end{lemma}
We need the following lemma:
\begin{lemma}\cite{KO}
Let $ u_{\alpha,\beta} $ be as in \eqref{ualphabeta} for some $ \alpha,\beta\in\mathbb{C}$. Let $ k\in \mathbb{N}$. Then
$$ \partial^k u_{\alpha,\beta}(z)=\frac{(\alpha+1)_k}{(1-z)^k}u_{\alpha,\beta}(z)+
\overline{z}g_k(z),\quad \vert z\vert<1, $$
where $ g_k(z)=\frac{(\alpha+1)_{k-1}}{(1-z)^{k-1}}g_1(z)+\partial^{k-1}g_{k-1}(z)\in C^\infty(\mathbb{D})$ for $k\geq 2 $ and $ g_1(z)=-\frac{\alpha+\beta+1}{1-\vert z\vert^2}u_{\alpha,\beta}(z), z\in \mathbb{D}$.
\end{lemma}
A generalization of this lemma is given below, see Lemma \ref{polinom}.

\begin{proposition}\label{DPfformula}
For $\varphi$ in $L^1(\mathbb T)$ and $z$ in $\mathbb D$ we have

\begin{align}\label{Dfz}
\lVert DP_{\alpha,\beta} [\varphi](z) \rVert = & \vert c_{\alpha,\beta} \vert 
\left[ \; \left\vert \int_{\mathbb T} \left( - \frac{(\alpha+\beta+1)\overline z }{1-\vert z\vert^2} + 
\frac{\alpha + 1}{1-z\overline \zeta}  \overline \zeta  \right)
u_{\alpha, \beta} (z \overline \zeta) \varphi(\zeta) dm(\zeta) \right\vert \right.         \nonumber      \\
+ & \left. \left\vert \int_{\mathbb T}
\left( -  \frac{(\alpha+\beta+1)z}{1 - \vert z\vert^2} +\frac{\beta + 1}{1- \overline z \zeta} \zeta \right)   u_{\alpha, \beta}
 (z \overline \zeta) \varphi(\zeta) dm(\zeta) \right\vert \; \right]          
\end{align}

and

\begin{equation}\label{Df0}
\lVert DP_{\alpha,\beta}[\varphi](0)\rVert = \vert c_{\alpha,\beta} \vert \left( \left\vert \int_{\mathbb T}(\alpha+1)\overline{\zeta}\varphi(\zeta)dm(\zeta)\right\vert +\left\vert \int_{\mathbb{T}}
(\beta+1){\zeta}\varphi(\zeta)dm(\zeta)\right\vert\right).
\end{equation}

\end{proposition}

\begin{proof}
Since $u_{\alpha, \beta}(0) = 1$ \eqref{Df0} follows from \eqref{Dfz}, so it suffices to prove \eqref{Dfz}.
We can use \eqref{normDf} and differentiation under the integral sign to obtain

\begin{align}
\lVert DP_{\alpha,\beta}[\varphi](z)\rVert = & \left\vert \frac{\partial}{\partial z} \int_{\mathbb T} c_{\alpha,\beta}
u_{\alpha, \beta} (z \overline \zeta) \varphi (\zeta) dm(\zeta) \right\vert + 
\left\vert \frac{\partial}{\partial \overline z} \int_{\mathbb T} c_{\alpha,\beta} u_{\alpha, \beta} (z \overline \zeta)
\varphi(\zeta) dm(\zeta) \right\vert          \nonumber   \\ 
= & \vert c_{\alpha,\beta}\vert \left( \left\vert \int_{\mathbb T} \left( -(\alpha+\beta+1) 
\frac{\overline z \zeta}{1-\vert z\vert^2} + \frac{\alpha + 1}{1-z\overline \zeta} \right) \overline \zeta 
u_{\alpha, \beta} (z \overline \zeta) \varphi(\zeta) dm(\zeta) \right\vert \right.         \nonumber      \\
+ & \left. \left\vert \int_{\mathbb T}
\left(-(\alpha+\beta+1) \frac{z \overline \zeta}{1 - \vert z\vert^2} +\frac{\beta + 1}{1-\overline{z}\zeta} \right) \zeta  u_{\alpha, \beta} (z \overline \zeta) \varphi(\zeta) dm(\zeta) \right\vert \right) ,     \nonumber
\end{align}
which suffices since $\vert \zeta \vert = 1$.
\end{proof}

For future reference we note the following identity
\begin{equation}\label{identity}
 -  \frac{(\alpha+\beta+1)\overline z }{1-\vert z\vert^2} + 
\frac{\alpha + 1}{1-z\overline \zeta}  \overline \zeta  = \frac{(\alpha + 1) \overline\zeta (1 - \overline z \zeta) - \beta \overline z (1 - \overline\zeta z)}{( 1- \vert z \vert^2)(1 - z \overline\zeta)}, \qquad \zeta \in \mathbb T, \quad z \in \mathbb D.
\end{equation}

It is convenient to use the following notation: $\lambda_1 = | \alpha + 1 |$, $\lambda_2 = | \beta + 1 |$ and

\begin{equation}\label{Ivarphi} 
I(\varphi) = \lambda_1 \left\vert \int_{\mathbb T}\overline{\zeta}\varphi(\zeta)dm(\zeta)\right\vert +\lambda_2\left\vert \int_{\mathbb{T}}
{\zeta}\varphi(\zeta)dm(\zeta)\right\vert.
\end{equation}
Note that $\lambda_1, \lambda_2 > 0$. We can rewrite \eqref{Df0} in the following form

\begin{equation}\label{Df0new}
\lVert DP_{\alpha,\beta}[\varphi](0)\rVert = \vert c_{\alpha,\beta} \vert I(\varphi), \qquad \varphi \in L^1(\mathbb T).
\end{equation}

The following elementary observation will be useful in our estimates:
\begin{equation}\label{brisati}
\max_{\vert \eta \vert =1}\vert z\overline{\eta}+w\eta \vert = \vert z\vert +\vert w \vert, \qquad z, w \in \mathbb C.
\end{equation} 

\begin{lemma}\label{Maxeta}
For every $\varphi \in L^1(\mathbb T)$ we have
\begin{equation}\label{maxeta}
I(\varphi) = \frac{1}{2\pi} \max_{\vert \eta \vert = 1}
\left\vert \int_{-\pi}^{\pi} e^{-it} \varphi (e^{it}) \left( \lambda_1  + \lambda_2 \frac{\overline \eta}{\eta} e^{2it}\right) dt \right\vert .
\end{equation}
\end{lemma}

\begin{proof}

\begin{align}
I(\varphi) = & \frac{1}{2\pi} \left\vert \int_{-\pi}^{\pi} \lambda_1 \varphi(e^{it})e^{-it}dt\right\vert + 
\frac{1}{2\pi} \left\vert \int_{-\pi}^{\pi} \lambda_2
\varphi(e^{it})e^{it}dt\right\vert       \qquad \mbox{\rm by} \quad  \eqref{Ivarphi}                  \nonumber \\
= & \frac{1}{2\pi}\max_{\vert \eta \vert = 1}
\left\vert\lambda_1{\eta} \int_{-\pi}^{\pi}(\varphi(e^{it})e^{-it}dt +\lambda_2\overline{\eta} \int_{-\pi}^{\pi}
\varphi(e^{it})e^{it}dt\right\vert            \qquad \mbox{\rm by} \quad  \eqref{brisati}                         \nonumber\\
=& \frac{1}{2\pi}\max_{\vert\eta\vert= 1}
\left\vert \int_{-\pi}^{\pi}\varphi(e^{it})\left(\lambda_1{\eta} e^{-it} +\lambda_2\overline{\eta} e^{it}\right)dt\right\vert\nonumber\\
=&\frac{1}{2\pi}\max_{\vert\eta\vert= 1}
\left\vert \int_{-\pi}^{\pi}e^{-it} \varphi (e^{it}) \left( \lambda_1  + \lambda_2 \frac{\overline{\eta}}{\eta} e^{2it}\right) dt\right\vert .    \nonumber
\qedhere \end{align}
\end{proof}

\begin{definition}\label{shLp}
Let $1 \leq p \leq \infty$. The space $h_{\alpha, \beta}^p({\mathbb D})$ consists of all $u \in h_{\alpha, \beta}(\mathbb D)$ such that
\begin{equation}\label{shpnorm}
\lVert u \rVert_{\alpha, \beta; p} = \sup_{0 \leq r < 1}     \left( \int_{\mathbb T} 
\vert u (r \zeta) \vert^p dm (\zeta) \right)^\frac{1}{p} < \infty, \qquad 1 \leq p < \infty,
\end{equation}
\begin{equation}\label{shinfty}
\lVert u \rVert_{\alpha, \beta; p}  = \sup_{z \in \mathbb D} |u(z)|<\infty, \qquad p = +\infty.
\end{equation}
\end{definition}

These spaces appeared for the first time, in a more general setting, in \cite{ABC}
The following theorem is a special case of results from \cite{ABC} and \cite{GMAhp}.

\begin{theorem}\label{ReprLp}
Let $u \in h^p_{\alpha, \beta} (\mathbb D)$, $ 1 < p\leq\infty $.
Then there is a unique $ \psi \in L^p ({\mathbb T}) $ such that $u = P_{\alpha, \beta}[\psi]$.
\end{theorem}

\section{Main results}

The next Lemma generalizes estimate \ref{L1est}.
\begin{lemma}\label{estuabp}
Let $ u_{\alpha,\beta} $ be as in \eqref{ualphabeta} and  $ 1\leq p<\infty$. Then
for  $ z\in \mathbb{D}$ 
\begin{equation}
\int_{\mathbb T}\vert u_{\alpha,\beta} (z \overline \zeta) \vert^p dm(\zeta)  \leq 
e^{\frac{p\pi}{2}
\vert \Im  \alpha-\Im \beta \vert} \frac{\Gamma(p(\Re \alpha + \Re \beta+2)-1)}{\Gamma^2\left(\frac{p(\Re \alpha + \Re \beta+2)}{2}\right)}(1-\vert z\vert^2)^{1-p}.
\end{equation}
\end{lemma}

\begin{proof}
Using estimate \eqref{procuab} we get 
$$ 
\int_{\mathbb{T}}\vert u_{\alpha,\beta}(z\overline \zeta)\vert^p dm(\zeta)\leq e^{\frac{p\pi}{2}
\vert \Im  \alpha-\Im \beta\vert}\int_{\mathbb{T}}\frac{(1-\vert z\vert^2)^{p(\Re \alpha + \Re \beta+1)}}{\vert 1-z\overline \zeta\vert^{p(\Re \alpha + \Re \beta+2)}}dm(\zeta),
$$
since $p(\Re \alpha + \Re \beta+2)>p\geq 1$ we can apply Proposition \ref{MA} to obtain \eqref{estuabp}.
\end{proof}

\begin{theorem}\label{Df0Lp}
Let $1 \leq p \leq +\infty$ and let $q$ be the exponent conjugate to $p$. Then 
\begin{equation}\label{dfoLp}
\sup_{\Vert \varphi \Vert_p \leq 1} \lVert DP_{\alpha,\beta} [\varphi] (0) \rVert =
\frac{\vert c_{\alpha,\beta}\vert}{2\pi}
\left(\int_{-\pi}^{\pi}\left\vert\vert \alpha+1\vert  +\vert \beta+1\vert e^{2it}\right\vert^qdt\right)^{\frac{1}{q}},
\end{equation}
with obvious modification in the case $p = 1$, $q = +\infty$.
\end{theorem}

\begin{proof}
Using \eqref{Df0new}, \eqref{maxeta} and duality  between $L^p$ and $L^q$ spaces we have

\begin{align*}
\sup_{\Vert \varphi \Vert_p \leq 1} \lVert DP_{\alpha,\beta} [\varphi] (0) \rVert =
 &  \vert c_{\alpha,\beta} \vert \sup_{\Vert \varphi \Vert_p \leq 1}
I(\varphi)             \nonumber       \\
= & \frac{ \vert c_{\alpha,\beta} \vert}{2\pi} \max_{\vert\eta\vert= 1} \sup_{\Vert \varphi \Vert_p \leq 1}
\left\vert \int_{-\pi}^{\pi}e^{-it} \varphi(e^{it})\left(\lambda_1  +\lambda_2\frac{\overline{\eta}}{\eta} e^{2it}\right)dt\right\vert       \\
 =   &  \frac{ \vert c_{\alpha,\beta} \vert }{2\pi}
\max_{\vert\eta\vert= 1}
\left(\int_{-\pi}^{\pi}\left\vert\lambda_1  +\lambda_2\frac{\overline{\eta}}{{\eta}} e^{2it}\right\vert^qdt\right)^{\frac{1}{q}}
\\  = & \frac{\vert c_{\alpha,\beta}\vert}{2\pi}
\left(\int_{-\pi}^{\pi}\left\vert\vert \alpha+1\vert  +\vert \beta+1\vert e^{2it}\right\vert^qdt\right)^{\frac{1}{q}}.
\end{align*}
The last equality rests on the fact that $\overline \eta / \eta = e^{i\theta}$ and, since the interval of integration is $(-\pi, \pi)$,
the penultimate integral is independent of $\eta$.
\end{proof}

If $p = +\infty$ then $q = 1$ and 
 \begin{align*}
   \int_{-\pi}^{\pi}\left\vert\lambda_1\  +\lambda_2 e^{2it}\right\vert dt = & 4
(\lambda_1+\lambda_2)\int_{0}^{\frac{\pi}{2}}
\sqrt{1-\frac{4\lambda_1\lambda_2}{(\lambda_1+\lambda_2)^2}\sin^2tdt}   \\
=&4(\lambda_1+\lambda_2)E\left(\frac{2\sqrt{\lambda_1\lambda_2}}{\lambda_1+\lambda_2}\right)   \\
=&4(\lambda_1+\lambda_2)\frac{\pi}{2}F\left(\frac 12,-\frac 12;1;\frac{4\lambda_1\lambda_2}{(\lambda_1+\lambda_2)^2}\right),
\end{align*}
where $E$ denotes complete elliptic integral of second kind, as in the next proposition, which follows from the above theorem.

\begin{proposition}\label{Df0infty}
For all $\varphi \in L^\infty (\mathbb T)$ we have
$$ \lVert DP_{\alpha,\beta}[\varphi](0)\rVert
\leq
D(\alpha, \beta)
\lVert\varphi\rVert_\infty$$
where the constant 
$$D(\alpha, \beta) = 
\frac{2\vert c_{\alpha,\beta}\vert (\vert \alpha+1\vert+\vert \beta+1\vert)E\left( 2 \frac{\sqrt{\vert( \alpha+1)(\beta+1)\vert}}{\vert \alpha+1\vert+\vert \beta+1\vert}\right)}{\pi}
$$
is the best possible.
\end{proposition}

\begin{remark}\label{remark1}
Proposition \ref{Df0infty} was obtained by specializing Theorem \ref{Df0Lp} to the case $p = +\infty$. Let us specialize further by  considering $T_\alpha$-harmonic functions where $\alpha > -1$ is real. In other words, we are considering $(\alpha/2,
\alpha/2)$-harmonic functions. Since $E(1) = 1$ and $c_{\alpha/2, \alpha/2} = \Gamma^2(1 + \alpha/2)/\Gamma ( 1 + \alpha)$ the above proposition gives us the following estimate
$$
 \lVert Df (0)\rVert  \leq  \frac{2(\alpha+2)}{\pi}\frac{\Gamma^2\left(\frac{\alpha}{2}+1\right)}{\Gamma(\alpha+1)}, \qquad f = P_{\alpha/2, \alpha/2}[\varphi], \quad \| \varphi \|_\infty \leq 1.
$$
This is in fact Theorem 5 from \cite{MA} stated for $T_\alpha$ harmonic functions $f : \mathbb D \rightarrow\mathbb D$. 
Indeed, any $T_\alpha$ harmonic function $f : \mathbb D \rightarrow \mathbb D$ has representation 
$ f = P_{\alpha/2, \alpha/2}[\varphi]$ with $ \| \varphi \|_\infty \leq 1$, see \cite{ABC} or \cite{GMAhp}. 
Of course, in the harmonic case we get classical Colonna's estimate at the origin with constant $4/\pi$.
\end{remark}

\begin{theorem}\label{Dfzp}
Let $ \varphi\in L^p(\mathbb{T}), 1\leq p\leq \infty. $ Then
$$ \lVert DP_{\alpha,\beta}[\varphi](z)\rVert\leq 
C_{\alpha,\beta,p}\frac{\vert \alpha+1\vert +\vert \beta+1\vert +\vert \alpha z\vert +\vert \beta z\vert}{(1-\vert z\vert^2)^{1+\frac 1p}}\lVert\varphi\rVert_{L^p(\mathbb{T})},$$ 
where $ C_{\alpha,\beta,p} $  is a constant that depends only on $\alpha ,\beta$ and $p$. 
\end{theorem}

\begin{proof}
Using estimate \eqref{Dfz} and identity \eqref{identity} we have

\begin{align}\label{estimateDPalphabeta}
\lVert DP_{\alpha,\beta}[\varphi](z)\rVert\leq &
\vert c_{\alpha,\beta}\vert\left( \int_{\mathbb{T}}\frac{(\vert\alpha+1\vert+\vert\beta \bar z\vert)\vert u_{\alpha, \beta}(z\overline{\zeta})\vert }{1-\vert z\vert^2}\vert \varphi(\zeta)\vert dm(\zeta)\right.\\
+& \left.\int_{\mathbb{T}}\frac{(\vert\beta+1\vert+\vert\alpha z \vert)\vert u_{\alpha, \beta}(z\overline{\zeta})\vert }{1-\vert z\vert^2}\vert \varphi(\zeta)\vert dm(\zeta)\right)\nonumber\\
 = & 
\vert c_{\alpha,\beta}\vert\frac{\vert \alpha+1\vert +\vert \beta+1\vert +\vert \alpha z\vert +\vert \beta z\vert}{(1-\vert z\vert^2)} \int_{\mathbb{T}}\vert u_{\alpha, \beta}(z\overline{\zeta})\vert \vert \varphi(\zeta)\vert dm(\zeta). \nonumber
\end{align}

Therefore, for $ 1<p\leq +\infty$, using H$\ddot{o}$lder inequality and Lemma \ref{estuabp}, we have   
 \begin{align*}
\lVert DP_{\alpha,\beta}[\varphi](z)\rVert 
\leq &
\vert c_{\alpha,\beta}\vert e^{\frac{\pi}{2}
\vert \Im  \alpha-\Im \beta\vert} \left(\frac{\Gamma(q(\Re \alpha + \Re \beta+2)-1)}{\Gamma^2\left(\frac{q(\Re \alpha + \Re \beta+2)}{2}\right)}\right)^\frac{1}{q}\frac{\vert \alpha+1\vert +\vert \beta+1\vert +\vert \alpha z\vert +\vert \beta z\vert}{(1-\vert z\vert^2)^{1+\frac 1p}}\lVert\varphi\rVert_{L^p(\mathbb{T})}.
\end{align*}

If $ p=1,$  let us notice that for fixed $ \zeta \in \mathbb{T}$ by \eqref{procuab}, we obtain

\begin{equation}\label{procenaualphabeta}\vert u_{\alpha,\beta}(z\overline \zeta)\vert \leq 
 e^{\frac{\pi}{2}
\vert \Im  \alpha-\Im \beta\vert}\frac{(1-\vert z\vert^2)^{\Re \alpha + \Re \beta+1}}{(1-\vert z\vert)^{\Re \alpha + \Re \beta+2}}
\leq e^{\frac{\pi}{2}
\vert \Im  \alpha-\Im \beta\vert}\frac{2^{\Re \alpha + \Re \beta+2}}{1-\vert z\vert^2} 
\end{equation}
and then using estimate \eqref{estimateDPalphabeta}, we have
$$ \lVert DP_{\alpha,\beta}[\varphi](z)\rVert    \leq 
C_{\alpha,\beta}
\vert c_{\alpha,\beta}\vert e^{\frac{\pi}{2}
\vert \Im  \alpha-\Im \beta\vert}    \frac{\lVert\varphi\rVert_{L^1(\mathbb{T})}}{(1-\vert z\vert^2)^{2}}. \qedhere
$$
\end{proof}

In  Remark \ref{remark1} we considered estimates of derivative at the origin.   Similarly let $f$ be a $ T_\alpha$-harmonic function such that $ \vert f(z)\leq M $ for all $ z\in \mathbb{D},$  where $ \alpha>-1$. Then the above theorem and integral representation by boundary values given by Theorem \ref{ReprLp} lead to the following estimate 

$$
 \lVert Df(z)\rVert\leq \frac{2+\alpha +\vert \alpha z\vert}{1-\vert z\vert^2} M, \qquad z \in \mathbb D,
$$
which appeared in \cite{CHENVuorinen}.

  Theorem \ref{Dfzp}, in the case $ p=+\infty$,   implies the following estimate:

\begin{equation}\label{Df0alpha}
 \lVert DP_{0,\alpha}[\varphi](z)\rVert\leq 
(\vert \alpha\vert+\alpha+2)\frac{\Gamma(\alpha +1)}{\Gamma^2\left(\frac{\alpha }{2}+1\right)}\cdot \frac{ \lVert \varphi\rVert_\infty }{1-\vert z\vert^2}, 
\quad \varphi \in L^\infty (\mathbb T), \quad \alpha>-1.
\end{equation}

Note that in Theorem \ref{Teorem LI} a stronger assumption is imposed on $\alpha$-harmonic function, namely it is assumed that it extends continuously to 
$\overline{\mathbb D}$.
Estimate \eqref{Df0alpha} is sharper than the estimate given in Theorem \ref{Teorem LI} from \cite{LIet} for any $ -1<\alpha<+\infty$. 

Also, the estimate \eqref{Df0alpha} gives better estimate on the order of growth of $ \lVert DP_{0,\alpha}[\varphi](z)\rVert $ as $ \vert z\vert\rightarrow 1 $ in the case $ -1<\alpha <0 $ than the estimate given in Lemma \ref{dva rada}.

For $ \alpha >0, $ by Legandre duplication formula 
 $ \Gamma(2z)=2^{2z-1}\pi^{-\frac 12}\Gamma(z)\Gamma(z+1/2), \Re z>0$ 
estimate \eqref{Df0alpha} gives
\begin{align}\label{Dfzalphainfty}  \lVert DP_{0,\alpha}[\varphi](z)\rVert\leq &
\frac{(\alpha+1)2^{\alpha+1}\pi^{-\frac 12}\Gamma\left(\frac{\alpha}{2}+\frac 12\right)}{\Gamma\left(\frac{\alpha}{2}+1\right)}\cdot\frac{ \lVert \varphi\rVert_\infty }{1-\vert z\vert^2}\nonumber\\
= & \frac{(\alpha+1)2^{\alpha+1} B\left(\frac{\alpha}{2}+\frac 12,\frac 12\right)}{\pi}\cdot\frac{ \lVert \varphi\rVert_\infty }{1-\vert z\vert^2}.
\end{align}

The function $ g(\alpha)= B\left(\frac{\alpha}{2}+\frac 12,\frac 12\right)/\pi $ is strictly decreasing for $  \alpha\in(0,+\infty) $  and $ g(0)=1.$ 
Therefore $ B\left(\frac{\alpha}{2}+\frac 12,\frac 12\right)<\pi $ and hence, the constant $ \pi^{-1}(\alpha+1)2^{\alpha+1} B\left(\frac{\alpha}{2}+\frac 12,\frac 12\right) $ in \eqref{Dfzalphainfty} is strictly smaller than the constant 
$ (\alpha+1)2^{\alpha+1} $ appearing in  Lemma \ref{dva rada}.

\begin{lemma}\label{polinom}
Let $ u_{\alpha,\beta} $ be as in \eqref{ualphabeta} for some $ \alpha,\beta\in\mathbb{C}$. Then for $ k\geq 0$ and $l\geq 0 $ we have
$$ \partial^k\overline{\partial}^l u_{\alpha,\beta}(z)= P_{k,l}\left(\frac{1}{1-\vert z\vert^2},\frac{1}{1-z},\frac{1}{1-\overline z}, f_1(z),\ldots, f_N(z)\right)u_{\alpha,\beta}(z), $$
where $ N=N_{k,l}, $ the functions $ f_1,\ldots, f_{N_{k,l}} $ are $ C^\infty $ on $\mathbb C$ and $ P_{k,l}$ is a polynomial whose total degree in $ (1-\vert z\vert^2)^{-1}, (1-z)^{-1},(1-\overline z)^{-1} $ is equal to $ k+l$. In other words, 
$$  P_{k,l}=P_{k,l}\left(\frac{1}{1-\vert z\vert^2},\frac{1}{1-z},\frac{1}{1-\overline z}\right) $$ 
is polynomial over an algebra $ C^\infty(\mathbb C) $ and $ \deg P_{k,l}=k+l$. 
\end{lemma}

\begin{proof}
We proceed by induction on $n=k+l$. The case $n = 1$  is Lemma \ref{partialu} from \cite{KO}. Now
 let $n > 1$ and assume that Lemma holds for $n-1.$ Let
$ u=(1-\vert z\vert^2)^{-1},v=(1-z)^{-1},w=(1-\overline z)^{-1}$. Then
\begin{align*}
\partial^{k+1}\overline{\partial}^l u_{\alpha,\beta}(z)=& 
\partial \left(P_{k,l}\left(\frac{1}{1-\vert z\vert^2},\frac{1}{1-z},\frac{1}{1-\overline z},f_1(z),\ldots, f_N(z)\right)u_{\alpha,\beta}(z)\right)\\
=& \left(\frac{\partial P_{k,l}}{\partial u}\frac{\overline{z}}{(1-\vert z\vert^2)^2}+\frac{\partial P_{k,l}}{\partial v}\frac{1}{(1-z)^2}+\sum_{i=1}^N\frac{\partial P_{k,l}}{\partial f_i}\frac{\partial f_i}{\partial z}\right.
\\+&\left.P_{k,l}\frac{\alpha+1}{1-z}-P_{k,l}\frac{(\alpha+\beta+1)\overline z}{1-\vert z\vert^2}\right) u_{\alpha,\beta}(z)\\
=& P_{k+1,l}u_{\alpha,\beta}(z)
\end{align*}
and similarly
\begin{align*}
\partial^{k}\overline{\partial}^{l+1} u_{\alpha,\beta}(z)=& 
\overline \partial \left(P_{k,l}\left(\frac{1}{1-\vert z\vert^2},\frac{1}{1-z},\frac{1}{1-\overline z},f_1(z),\ldots, f_N(z)\right)u_{\alpha,\beta}(z)\right)\\
=& \left(\frac{\partial P_{k,l}}{\partial u}\frac{{z}}{(1-\vert z\vert^2)^2}+\frac{\partial P_{k,l}}{\partial w}\frac{1}{(1-\overline z)^2}+\sum_{i=1}^N\frac{\partial P_{k,l}}{\partial f_i}\frac{\partial f_i}{\partial \overline z}\right.
\\+&\left.P_{k,l}\frac{\beta+1}{1-\overline z}-P_{k,l}\frac{(\alpha+\beta+1) z}{1-\vert z\vert^2}\right) u_{\alpha,\beta}(z)\\
=&  P_{k,l+1}u_{\alpha,\beta}(z).
\end{align*}
Clearly $ P_{k+1,l}, P_{k,l+1}$ are the polynomials whose total degree in $ (1-\vert z\vert^2)^{-1}, (1-z)^{-1},(1-\overline z)^{-1} $ is equal to $ k+l+1$.
\end{proof}

\begin{remark}
Ako je $ k=l$ onda je $ P_{k,l}=P_{k,l}\left(\frac{1}{1-\vert z\vert^2},\frac{1}{1-z},\frac{1}{1-\overline z}\right) $ polinom po tim promjenljivim sa konstantnim koeficijentima. (nema ni $ z $ ni $ \overline z$.)
\end{remark}

\begin{proposition}\label{visiuab}
Let $ u_{\alpha,\beta} $ be as in \eqref{ualphabeta} for some $ \alpha,\beta\in\mathbb{C}$. Then
\begin{equation}\label{estkluab}
\left \vert \partial^k\overline{\partial}^l u_{\alpha,\beta}(z)\right\vert \leq C_{\alpha,\beta,k,l}\frac{\vert u_{\alpha,\beta}(z)\vert}{(1-\vert z\vert^2)^{k+l}}\quad z\in \mathbb{D}.
\end{equation}
\end{proposition}

\begin{proof}
The proposition follows from  Lemma \ref{polinom} and the fact that $ \frac 12(1-\vert z\vert^2)\leq 1-\vert z \vert \leq \vert 1-z\vert=\vert 1-\overline z\vert $ for all $ z\in \mathbb{D}$.
\end{proof}
The next theorem extends Theorem \ref{Dfzp} to higher order derivatives. 

\begin{theorem}
Let $ \varphi\in L^p(\mathbb{T}), 1\leq p\leq +\infty$. Then
\begin{equation}\label{estkluab}
\left \vert \partial^k\overline{\partial}^l u(z)\right\vert \leq C_{\alpha,\beta,k,l,p}\frac{\lVert \varphi\rVert_p}{(1-\vert z\vert^2)^{k+l+\frac 1p}}, \qquad  u=P_{\alpha,\beta}[\varphi].
\end{equation}
\end{theorem}

\begin{proof}
By Proposition \ref{visiuab} we have  

$$
\left \vert \partial^k\overline{\partial}^l u(z)\right\vert =
\left \vert c_{\alpha,\beta}\int_{\mathbb{T}}\partial^k\overline{\partial}^l u_{\alpha,\beta}(z\overline \zeta)\varphi(\zeta)dm(\zeta)\right\vert
\leq  C_{\alpha,\beta,k,l}\int_{\mathbb{T}}\frac{\vert u_{\alpha,\beta}(z\overline \zeta)\vert}{(1-\vert z\vert^2)^{k+l}}\vert \varphi(\zeta)\vert dm(\zeta), \qquad z \in \mathbb D.
$$

If $ p=+\infty$, then
 
\begin{align*}
\int_{\mathbb{T}}\frac{\vert u_{\alpha,\beta}(z\overline \zeta)\vert}{(1-\vert z\vert^2)^{k+l}}\vert \varphi(\zeta)\vert dm(\zeta)\leq & 
\frac{\lVert \varphi\rVert_\infty}{(1-\vert z\vert^2)^{k+l}}
\int_{\mathbb{T}}\vert u_{\alpha,\beta}(z\overline \zeta)\vert dm(\zeta) \\
\leq & e^{\frac{\pi}{2}
\vert \Im  \alpha-\Im \beta\vert}\frac{\Gamma(\Re \alpha + \Re \beta+1)}{\Gamma^2\left(\frac{\Re \alpha + \Re \beta+2}{2}\right)}\frac{\lVert \varphi\rVert_\infty}{(1-\vert z\vert^2)^{k+l}}, 
\end{align*}
and, if $ p=1, $ using estimate \eqref{procenaualphabeta}, we have
\begin{align*}
\int_{\mathbb{T}}\frac{\vert u_{\alpha,\beta}(z\overline \zeta)\vert}{(1-\vert z\vert^2)^{k+l}}\vert \varphi(\zeta)\vert dm(\zeta)\leq & \frac{C_{\alpha,\beta}}{(1-\vert z\vert^2)^{k+l+1}}\int_{\mathbb{T}}\vert \varphi(\zeta)\vert dm(\zeta)=
\frac{C_{\alpha,\beta}}{(1-\vert z\vert^2)^{k+l+1}}\lVert \varphi\rVert_{L^1(\mathbb{T})}.
\end{align*}
For $ 1<p<\infty$ using H$\ddot{o} $lder inequality and Lemma \ref{estuabp}, we obtain

\begin{align*}
\int_{\mathbb{T}}\frac{\vert u_{\alpha,\beta}(z\overline \zeta)\vert}{(1-\vert z\vert^2)^{k+l}}\vert \varphi(\zeta)\vert dm(\zeta)\leq &
\frac{\lVert \varphi\rVert_{L^p(\mathbb{T})}}{(1-\vert z\vert^2)^{k+l}}
\left(\int_{\mathbb{T}}\vert u_{\alpha,\beta}(z\overline \zeta)\vert^q dm(\zeta) \right)^{\frac{1}{q}}\\
 \leq &\frac{C_{\alpha,\beta,p}\lVert \varphi\rVert_{L^p(\mathbb{T})}}{(1-\vert z\vert^2)^{k+l}} \left(1-\vert z\vert^2\right)^{\frac{1}{q}-1}\\
=& \frac{C_{\alpha,\beta,p}\lVert \varphi\rVert_{L^p(\mathbb{T})}}{(1-\vert z\vert^2)^{k+l+\frac{1}{p}}}
\end{align*}

where $ C_{\alpha,\beta,p}=e^{\frac{\pi}{2}
\vert \Im  \alpha-\Im \beta\vert}\left(\frac{\Gamma(q(\Re \alpha + \Re \beta+2)-1)}{\Gamma^2\left(\frac{q(\Re \alpha + \Re \beta+2)}{2}\right)}\right)^{1-\frac{1}{p}}$.
\end{proof}

\end{document}